\documentclass{amsart}
\usepackage{amsmath,amssymb,apalike,amsthm}

\newtheorem{theorem}{Theorem}
\newtheorem{lemma}[theorem]{Lemma}
\newtheorem{proposition}[theorem]{Proposition}
\newtheorem{corollary}[theorem]{Corollary}

\newtheorem{remark}[theorem]{Remark}
\newtheorem{example}[theorem]{Example}

\begin{document}
\title[Localization of injective modules]{Localization of injective modules over arithmetical rings}
\author{Fran\c cois Couchot}
\address{Laboratoire de Math\'ematiques Nicolas Oresme, CNRS UMR
  6139,
D\'epartement de math\'ematiques et m\'ecanique,
14032 Caen cedex, France}
\email{couchot@math.unicaen.fr} 



\begin{abstract}  It is proved that localizations of 
injective $R$-modules  of finite Goldie dimension are  injective if $R$ is an arithmetical ring satisfying the following condition: for every maximal ideal $P$, $R_P$ is either coherent or not semicoherent. If, in addition, each finitely generated $R$-module has finite Goldie dimension, then localizations of finitely injective $R$-modules are finitely injective too. Moreover, if $R$ is a Pr\"ufer domain of finite character, localizations of 
injective $R$-modules  are  injective. 
\end{abstract}
\maketitle

This is a sequel and a complement of \cite{Cou06}. If $R$ is a noetherian or hereditary ring, it is well known that localizations of injective $R$-modules are injective. By \cite[Corollary 8]{Cou06} this property holds if $R$ is a h-local Pr\"ufer domain. However \cite[Example 1]{Cou06}
shows that this result is not generally true. E. C. Dade was probably the first  to study localizations of injective modules. By \cite[Theorem 25]{Dad81}, there exist a ring $R$, a multiplicative subset $S$ and an injective module $G$ such that $S^{-1}G$ is not injective. In this example we can choose $R$ to be a coherent domain.

The aim of this paper is to study localizations of injective modules over arithmetical rings. We deduce from \cite[Theorem 3]{Cou06} the two following results: any localization of an injective $R$-module of finite Goldie dimension  is injective if and only if any localization at a maximal ideal of $R$ is either coherent or non-semicoherent (Theorem~\ref{T:gold}) and each localization of any injective module over a Pr\"ufer domain of finite character is injective (Theorem~\ref{T:finchar}). Moreover, if any localization at a maximal ideal of $R$ is either coherent or non-semicoherent, and if each finitely generated $R$-module has a finite Goldie dimension, then each localization of any finitely injective $R$-module is finitely injective.

In this paper all rings are associative and commutative with unity and
all modules are unital. 
A module is said to be
\textit{uniserial} if its  submodules are linearly ordered by inclusion.
A ring $R$ is a \textit{valuation ring} if it is uniserial as
$R$-module and $R$ is \textit{arithmetical} if $R_P$ is a
valuation ring for every maximal ideal $P.$ An arithmetical domain $R$ is said to be \textit{Pr\"ufer}. We say that a module $M$ is of {\it Goldie dimension} $n$ if and only if its injective hull $\mathrm{E}(M)$ is a direct
sum of $n$ indecomposable injective modules. 
 We say that a domain $R$ is \textit{of finite
  character} if every non-zero element is contained in  finitely
many maximal ideals.

As in \cite{RaRa73}, a module $M$ over a ring $R$ is said to be {\it finitely injective} if every homomorphism $f:A\rightarrow M$ extends to $B$ whenever $A$ is a finitely generated submodule of an arbitrary $R$-module $B$.

As in \cite{Mat85} a ring $R$ is said to be \textit{semicoherent} if $\mathrm{Hom}_R(E,F)$ is a submodule of a flat $R$-module for any pair of injective $R$-modules $E,\ F$. An $R$-module $E$ is {\it FP-injective} if 
$\hbox{Ext}_R^1 (F, E) = 0$  for any finitely presented $R$-module $F,$ and $R$ is {\it self FP-injective} if $R$ is 
FP-injective as $R$-module.  We recall that a module $E$ is FP-injective if and only if it is a pure submodule
of every overmodule. If each injective $R$-module is flat we say that $R$ is an \textit{IF-ring}. By \cite[Theorem 2]{Col75}, $R$ is an IF-ring if and only if it is coherent and self FP-injective.

\medskip

We begin by some results on semicoherent rings.

\begin{proposition}
\label{P:FPinj} Let $R$ be a self FP-injective ring. Then $R$ is coherent if and only if it is semicoherent.
\end{proposition}
\begin{proof} If $R$ is coherent then, by \cite[Theorem~XIII.6.4(b)]{FuSa01}, $\mathrm{Hom}_R(E,F)$ is flat for any pair of injective modules $E,\ F$; so, $R$ is semicoherent. Conversely, let $E$ be the injective hull of $R$. Since $R$ is a pure submodule of $E$, then, for each injective $R$-module $F$, the following sequence is exact:
\[0\rightarrow\mathrm{Hom}_R(F\otimes_RE/R,F)\rightarrow\mathrm{Hom}_R(F\otimes_RE,F)\rightarrow\mathrm{Hom}_R(F\otimes_RR,F)\rightarrow 0.\]
By using the natural isomorphisms $\mathrm{Hom}_R(F\otimes_RB,F)\cong\mathrm{Hom}_R(F,\mathrm{Hom}_R(B,F))$ and $F\cong\mathrm{Hom}_R(R,F)$ we get the following exact sequence:
\[0\rightarrow\mathrm{Hom}_R(F,\mathrm{Hom}_R(E/R,F))\rightarrow\mathrm{Hom}_R(F,\mathrm{Hom}_R(E,F))\rightarrow\mathrm{Hom}_R(F,F)\rightarrow 0.\]
So, the identity map on $F$ is the image of an element of $\mathrm{Hom}_R(F,\mathrm{Hom}_R(E,F))$.
Consequently the following sequence splits:
\[0\rightarrow\mathrm{Hom}_R(E/R,F)\rightarrow\mathrm{Hom}_R(E,F)\rightarrow F\rightarrow 0.\]
It follows that $F$ is a summand of a flat module. So, $R$ is an IF-ring. 
\end{proof}

\begin{corollary}
\label{C:semicoh} Let $R$ be a ring such that its ring of quotients $Q$ is self FP-injective. Then $R$ is semicoherent if and only if $Q$ is coherent.
\end{corollary}
\begin{proof} If $R$ is semicoherent, then so is $Q$ by \cite[Proposition 1.2]{Mat85}. From Proposition~\ref{P:FPinj} we deduce that $Q$ is coherent. Conversely, let $E$ and $F$ be injective $R$-modules. It is easy to check that the multiplication by a regular element of $R$ in $\mathrm{Hom}_R(E,F)$ is injective. So, $\mathrm{Hom}_R(E,F)$ is a submodule of the injective hull of $Q\otimes_R\mathrm{Hom}_R(E,F)$ which is flat over $Q$ and $R$ because $Q$ is an IF-ring. 
\end{proof}

\medskip

From Corollary~\ref{C:semicoh} and \cite[Theorem II.11]{Cou03}  we deduce the following:
\begin{corollary}
\label{C:ValSemiCoh} Let $R$ be a valuation ring. Denote by $Z$ its subset of zerodivisors which is a prime ideal. Assume that $Z\ne 0$. Then the following conditions are equivalent:
\begin{enumerate}
\item $R$ is semicoherent;
\item $R_Z$ is an IF-ring;
\item $Z$ is not a flat $R$-module.
\end{enumerate} 
\end{corollary}

\medskip

From Corollary~\ref{C:ValSemiCoh} and \cite[Theorem 3]{Cou06} we deduce the following:

\begin{corollary}
\label{C:InjVal} Let $R$ be a valuation ring of maximal ideal $P$. Then  the following conditions are equivalent:
\begin{enumerate}
\item $R$ is either coherent or non-semicoherent;
\item for each multiplicative subset $S$ of $R$ and for each injective  $R$-module $E$, $S^{-1}E$ is injective;
\item for each multiplicative subset $S$ of $R$ and for each FP-injective $R$-module $E$, $S^{-1}E$ is FP-injective;
\item $(\mathrm{E}_R(R/P))_Z$ is FP-injective.
\end{enumerate}  
\end{corollary}
\begin{proof} $(1)\Rightarrow (2)$. Since $R$ is a valuation ring, $R\setminus S$ is a prime ideal. If $R$ is coherent then either $Z=0$ or $Z=P$. In the first case $Z$ is flat and in the second $E$ is flat. So we conclude by \cite[Theorem 3]{Cou06}. If $R$ is not semicoherent then $Z$ is flat. We conclude in the same way. 

$(2)\Rightarrow (3)$. $E$ is a pure submodule of its injective hull $H$. Then $S^{-1}E$ is a pure submodule of $S^{-1}H$ which is injective by $(2)$. So, $S^{-1}E$ is FP-injective.

$(3)\Rightarrow (4)$ is obvious.

$(4)\Rightarrow (1)$. Suppose that $Z$ is not flat. If $R$ is not coherent, then, by  \cite[Theorem II.11]{Cou03}, $Z\ne 0$ and $Z\ne P$, and $R$ is not self FP-injective.  Let $E=\mathrm{E}_R(R/P)$. By \cite[Proposition 2.4]{Cou82} $E$ is not flat. Now, we do as in the last part of the proof of \cite[Theorem 3]{Cou06} to show that $E_Z$ is not FP-injective. This contradicts $(4)$. The proof is now complete. 
\end{proof}

\begin{theorem} \label{T:gold} For any arithmetical ring $R$ the following conditions are equivalent:
\begin{enumerate}
\item for each maximal ideal $P$, $R_P$ is either coherent or non-semicoherent;
\item for each multiplicative subset $S$ and for each injective  $R$-module $G$ of finite Goldie dimension, $S^{-1}G$ is injective;
\item for each multiplicative subset $S$  and for each FP-injective $R$-module $G$ of finite Goldie dimension, $S^{-1}G$ is FP-injective;
\item for each maximal ideal $P$, $Q(R_P)\otimes_R\mathrm{E}_R(R/P)$ is FP-injective, where $Q(R_P)$ is the ring of fractions of $R_P$.
\end{enumerate}  
\end{theorem}
\begin{proof} $(1)\Rightarrow (2)$. $G$ is a finite direct sum of indecomposable injective modules. We may assume that $G$ is indecomposable.  Since $\mathrm{End}_RG$ is local, there exists a maximal ideal $P$ such that  $G$ is a module over $R_P$. If $S'$ is the multiplicative subset of $R_P$ generated by $S$, then $S^{-1}G=S'^{-1}G$. We conclude that $S^{-1}G$ is injective by Corollary~\ref{C:InjVal}. 

We show  $(2)\Rightarrow (3)$ as in the proof of Corollary~\ref{C:InjVal}, and $(3)\Rightarrow (4)$ is obvious. 

$(4)\Rightarrow (1)$ is an immediate consequence of Corollary~\ref{C:InjVal}.
\end{proof}

\begin{remark}
\textnormal{If $R$ is an arithmetical ring which is coherent or reduced, then $R$ satisfies the conditions of Theorem~\ref{T:gold}.}
\end{remark}

\begin{corollary} \label{C:gold}
Let $R$ be an arithmetical ring satisfying the following two conditions:
\begin{itemize}
\item[(a)] for each maximal ideal $P$, $R_P$ is either coherent or non-semicoherent;
\item[(b)] every finitely generated $R$-module has a finite Goldie dimension.
\end{itemize}
Then, for each multiplicative subset $S$ and for each finitely injective (respectively FP-injective) $R$-module $G$ , $S^{-1}G$ is finitely injective (respectively FP-injective).

Moreover, if $R_P$ is coherent for each maximal ideal $P$ then $R$ is coherent too.
\end{corollary}
\begin{proof} Let $M$ be a finitely generated $S^{-1}R$-submodule of $S^{-1}G$. There exists a finitely generated submodule $M'$ of $G$ such that $M=S^{-1}M'$. If $G$ is finitely injective, by \cite[Proposition 3.3]{RaRa73} it contains an injective hull $E$ of $M'$. Then $E$ has finite Goldie dimension. By Theorem~\ref{T:gold} $S^{-1}E$ is injective. It contains $M$ and it is contained in $S^{-1}G$. By using again \cite[Proposition 3.3]{RaRa73} we conclude that $S^{-1}G$ is finitely injective. 

If $G$ is FP-injective, it is a pure submodule of its injective hull $H$. Then $S^{-1}G$ is a pure submodule of $S^{-1}H$ which is finitely injective. So, $S^{-1}G$ is FP-injective. 

The last assertion is an immediate consequence of \cite[Th\'eor\`eme 1.4]{Cou82}.
\end{proof}

\begin{remark}
\textnormal{If $R$ is an arithmetical ring satisfying the condition (b) of Corollary~\ref{C:gold} then $\mathrm{Min}\ R/A$ is finite for each ideal $A$: we may assume that $A=0$ and $R$ is reduced; its total ring of quotient is Von Neumann regular by \cite[Proposition 2 p. 106]{Lam66} and semisimple by \cite[Proposition 2 p. 103]{Lam66}; it follows that $\mathrm{Min}\ R$ is finite. However, the converse doesn't hold. For instance, let $R =\{{\binom{d\ \ e}{0\ \ d}}\mid d\in \mathbb{Z}, e\in\mathbb{Q}/\mathbb{Z}\}$ be the trivial extension of $\mathbb{Z}$ by $\mathbb{Q}/\mathbb{Z}$.  Then $N=\{{\binom{0\ \ e}{0\ \ 0}}\mid e\in\mathbb{Q}/\mathbb{Z}\}$ is the only minimal prime. For each prime integer $p$, the localization $R_{(p)}$ is the trivial extension of $\mathbb{Z}_{(p)}$ by $\mathbb{Q}/\mathbb{Z}_{(p)}$. So it is a valuation ring. Consequently $R$ is arithmetical. But $N\cong\mathbb{Q}/\mathbb{Z}$ is an infinite direct sum, whence $R$ is not a module of finite Goldie dimension.}
\end{remark}
\medskip

By \cite[Corollary 8]{Cou06}, if $R$ is a h-local Pr\"ufer domain, all localizations of injective $R$-modules are injective. Now, we extend this result to each Pr\"ufer domain of finite character. A such ring satisfies condition(b) of Corollary~\ref{C:gold}. But $\mathbb{Z}+X\mathbb{Q}[[X]]$ is an example showing that the converse doesn't hold.
\begin{lemma}
\label{L:prodinj} Let $R$ be a Pr\"ufer domain of finite character. For each maximal ideal $P$, let $F_{(P)}$ be an  injective $R_P$-module and let $F=\prod_{P\in\mathrm{Max}\ R}F_{(P)}$. Then $S^{-1}F$ is injective for every multiplicative subset $S$ of $R$.
\end{lemma}
\begin{proof} Let $T_{(P)}$ be the torsion submodule of $F_{(P)}$, let $G_{(P)}=F_{(P)}/T_{(P)}$, let $T=\prod_{P\in\mathrm{Max}\ R}T_{(P)}$ and let $G=\prod_{P\in\mathrm{Max}\ R}G_{(P)}$. Then $G$ is torsion-free and $F\cong T\oplus G$. It is obvious that $S^{-1}G$ is injective. Let $T'=\oplus_{P\in\mathrm{Max}\ R}T_{(P)}$. Since $R$ has finite character, it is easy to check that $T'$ is the torsion submodule of $T$. So, $T'$ is injective and $S^{-1}(T/T')$ is injective. For each maximal ideal $P$, $S^{-1}T_{(P)}$ is injective by \cite[Theorem 3]{Cou06}. Since $S^{-1}T'$ is the torsion submodule of $\prod_{P\in\mathrm{Max}\ R}S^{-1}T_{(P)}$, we successively deduce the injectivity of $S^{-1}T'$ and $S^{-1}T$. 
\end{proof}

\begin{theorem} \label{T:finchar} Let $R$ be a Pr\"ufer domain of finite character. Then, for each  injective module $G$, $S^{-1}G$ is injective for every multiplicative subset $S$ of $R$.
\end{theorem}
\begin{proof} Let $E=\prod_{P\in\mathrm{Max}\ R}\mathrm{E}_R(R/P)$ and let $F=\mathrm{Hom}_R(\mathrm{Hom}_R(G,E),E)$. Then $E$ is an injective cogenerator and $G$ is isomorphic to a summand of $F$. Since $R$ is coherent, $\mathrm{Hom}_R(G,E)$ is flat by \cite[Theorem~XIII.6.4(b)]{FuSa01}. Thus $F$ is injective. We put $F_{(P)}=\mathrm{Hom}_R(\mathrm{Hom}_R(G,E),\mathrm{E}_R(R/P))$. Then $F_{(P)}$ is an injective $R_P$-module and $F\cong\prod_{P\in\mathrm{Max}\ R}F_{(P)}$. By Lemma~\ref{L:prodinj} $S^{-1}F$ is injective. We conclude that $S^{-1}G$ is injective too. 
\end{proof}

\begin{corollary} \label{C:semiInj} 
Let $R$ be a semilocal Pr\"ufer domain. Then, for each  injective module $G$, $S^{-1}G$ is injective for every multiplicative subset $S$ of $R$.
\end{corollary}

The following example shows that the finite character is not a necessary condition in order that localizations of injective modules at multiplicative subsets are still injective.

\begin{example}
\textnormal{Let $R$ be the ring defined in \cite[Example 39]{Hut81}. Then $R$ is a Pr\"ufer domain which is not of finite character. But, since $R$ is the union of a countable family of principal ideal subrings, it is easy to check that, for any multiplicative subset $S$, $R$ satisfies \cite[Condition 14]{Dad81}. So, for each injective module $G$, $S^{-1}G$ is injective by \cite[Theorem 15]{Dad81}.}

\textnormal{Here another example communicated to me by L. Salce. Take $R$ constructed as in Chapter III, Example 5.5 of \cite{FuSa01}, which is a classical example by Heinzer-Ohm of almost Dedekind domain not of finite character. If you start with a countable field $K$, then $R$ is countable, hence conditions (14a) and (14c) of \cite{Dad81} are satisfied. Condition (14b) must be checked only for $I$ principal ideal, and it is easy to see that it holds true.} 
\end{example}

\medskip
Consequently, the following question is unsolved:

\textbf{Open question:}  {\it characterize the Pr\"ufer domains such that localizations of injective at multiplicative subsets are still injective.}

\end{document}